\documentclass[11pt]{amsart}

\usepackage[margin=1in]{geometry}
\usepackage{amsmath,amssymb,amsthm}
\usepackage{graphicx}
\usepackage{microtype}
\usepackage{subcaption}
\usepackage{placeins}
\usepackage{xcolor}
\usepackage{hyperref}

\hypersetup{
	colorlinks=true,
	linkcolor=blue!55!black,
	citecolor=blue!55!black,
	urlcolor=blue!55!black,
	pdftitle={Anticoncentration of the Permanent in Ginibre Ensembles},
	pdfauthor={Frederic Koehler and Pui Kuen Leung}
}

\newtheorem{theorem}{Theorem}[section]
\newtheorem{proposition}[theorem]{Proposition}
\newtheorem{lemma}[theorem]{Lemma}
\newtheorem{corollary}[theorem]{Corollary}
\theoremstyle{definition}
\newtheorem{definition}[theorem]{Definition}
\theoremstyle{remark}
\newtheorem{remark}[theorem]{Remark}

\newcommand{\C}{\mathbb C}
\newcommand{\R}{\mathbb R}
\newcommand{\Hh}{\mathbb H}
\newcommand{\F}{\mathbb F}
\newcommand{\E}{\mathbb E}
\newcommand{\Pp}{\mathbb P}
\newcommand{\ii}{\mathrm i}
\newcommand{\dd}{\,\mathrm d}
\DeclareMathOperator{\per}{per}
\DeclareMathOperator{\Sdet}{Sdet}
\DeclareMathOperator{\Cdet}{Cdet}
\DeclareMathOperator{\RePart}{Re}

\title[Anticoncentration of the permanent in Ginibre ensembles]
{Anticoncentration of the Permanent in Ginibre Ensembles}
\author{Frederic Koehler}
\author{Pui Kuen Leung}
\date{}

\begin{document}
	
	\begin{abstract}
		Let $\mathbb K\in\{\R,\C,\Hh\}$, put
		$\beta=\dim_{\R}\mathbb K$, and let $G_n^{\mathbb K}$ be an $n\times n$
		matrix with i.i.d. standard $\mathbb K$-Gaussian entries, a standard
		$\mathbb K$-Ginibre matrix.  We prove that the normalized row-ordered
		permanent
		$W_n^{\mathbb K}=\per_{\mathbb K}G_n^{\mathbb K}/\sqrt{n!}$
		has a radial density $p_n^{\mathbb K}$ satisfying
		\[
		\|p_n^{\mathbb K}\|_\infty=p_n^{\mathbb K}(0)
		\lesssim_\beta n^{(\beta+2)/4},
		\qquad
		\sup_{z\in\mathbb K}
		\Pp(|W_n^{\mathbb K}-z|\le\varepsilon)
		\lesssim_\beta n^{(\beta+2)/4}\varepsilon^\beta.
		\]
		In particular, for $\mathbb K=\C$, this resolves the Permanent
		Anticoncentration Conjecture of Aaronson and Arkhipov.	
		The proof compares the squared Gaussian permanent with the squared (Study)
		determinant in Laplace-transform order.
	\end{abstract}
	
	\maketitle
	
	\section{Introduction}
	
	Let
	\[
	\mathbb K\in\{\R,\C,\Hh\},
	\qquad
	\beta=\dim_{\R}\mathbb K\in\{1,2,4\}.
	\]
	For $A\in\mathbb K^{n\times n}$, define the row-ordered permanent
	\[
	\per_{\mathbb K}A
	:=\sum_{\sigma\in S_n}
	a_{1,\sigma(1)}a_{2,\sigma(2)}\cdots a_{n,\sigma(n)}.
	\]
	Over $\R$ and $\C$ this is the ordinary permanent, and we usually omit the
	subscript.  Over $\Hh$ the order of multiplication is essential; this
	is also known as the Cayley permanent \cite{ArvindSrinivasan2010,EngelsRao2020}, and we write it as
	$\per_{\Hh}A$.
	
	Following Dyson's unitary, symplectic, and orthogonal classes of Hermitian random matrices, Ginibre's foundational paper \cite{Ginibre1965} introduced
	non-Hermitian Gaussian ensembles in the same three cases: complex,
	quaternion, and real.
	A standard $\mathbb K$-Gaussian is a scalar whose real coordinates are
	independent $N(0,1/\beta)$ variables; in particular, $\E|Z|^2=1$.
	Let $G_n^{\mathbb K}$ be a matrix with i.i.d. standard
	$\mathbb K$-Gaussian entries, yielding the standard
	$\mathbb K$-Ginibre matrix.

	Set
	\[
	W_n^{\mathbb K}
	:=\frac{\per_{\mathbb K}G_n^{\mathbb K}}{\sqrt{n!}}.
	\]
	The Permanent Anticoncentration Conjecture (PACC), introduced by Aaronson
	and Arkhipov in connection with BosonSampling \cite{AaronsonArkhipov2013},
	concerns $\mathbb K=\C$ and asserts that there is a polynomial $p$ such that,
	for all $n\ge1$ and $\delta>0$,
	\[
	\Pp\left(|\per G_n|<\frac{\sqrt{n!}}{p(n,1/\delta)}\right)<\delta.
	\]
	
	We briefly review some previous work on permanent anticoncentration.
	Tao and Vu proved the typical-scale estimate
	$|\per A_n|=n^{n/2+o(n)}$ with probability $1-o(1)$ for i.i.d.\ Bernoulli
	matrices \cite{TaoVu2009}; their
	argument also applies to Gaussian entries \cite{TaoMathOverflow2010}.  Kwan and Sauermann proved the analogous typical-scale
	estimate for random symmetric Bernoulli matrices \cite{KwanSauermann2022}.
	Ingram and Razborov proved a superpolynomial lower bound for
	the number of values attained by the permanent on sign matrices
	\cite{IngramRazborov2026}, and the exponential point-mass bound of Hunter, Kwan,
	and Sauermann under atom-size assumptions \cite{HunterKwanSauermann2025} gives
	exponential range and resolves the question raised by Ingram and Razborov.
	Subsequent work on Gaussian permanents has studied moments, numerical behavior,
	representation-theoretic formulas, and zeros
	\cite{KoehlerLeung2026,LundowMarkstrom2022,Nezami2021}.  
	
	Our work focuses on the Gaussian case. We prove the
	following.
	
	\begin{theorem}[Gaussian permanent anticoncentration]
		\label{thm:pacc-headline}
		For every $\mathbb K\in\{\R,\C,\Hh\}$ and $n\ge1$,
		$W_n^{\mathbb K}$ has a radial density $p_n^{\mathbb K}$ satisfying
		\begin{equation}\label{eq:gaussian-density-headline}
			\|p_n^{\mathbb K}\|_\infty
			=p_n^{\mathbb K}(0)
			\lesssim_\beta n^{(\beta+2)/4}.
		\end{equation}
		Consequently, for every $\varepsilon>0$,
		\begin{equation}\label{eq:gaussian-small-ball-headline}
			\sup_{z\in\mathbb K}
			\Pp(|W_n^{\mathbb K}-z|\le\varepsilon)
			\lesssim_\beta n^{(\beta+2)/4}\varepsilon^\beta.
		\end{equation}
	\end{theorem}
	
	For $\mathbb K=\C$, the small-ball estimate reads
	\[
	\sup_{z\in\C}
	\Pp\bigl(|\per G_n-z|\le\varepsilon\sqrt{n!}\bigr)
	\lesssim n\varepsilon^2.
	\]
	In particular, PACC holds; see \cite{AaronsonArkhipov2013} for its role in
	hardness for classical simulation of BosonSampling and
	\cite[Section~9.2]{KoehlerLeung2026} for an application to zeros of biased
	Gaussian permanents.
	Our result also proves the real-Gaussian version of PACC, which has appeared in several places; see, for example,
	\cite[Conjecture~2]{ChabaudEtAl2017} and
	\cite[Conjecture~9]{BoulandDattaFeffermanHernandez2025}.  More generally,
	the Gaussian perturbation result in
	Section~\ref{sec:gaussian-perturbations} shows that, for every
	deterministic $A\in\R^{n\times n}$,
	\[
	\sup_{w\in\R}
	\Pp\bigl(
	|\per(G_n^{\R}+A)-w|\le\varepsilon\sqrt{n!}
	\bigr)
	\lesssim n^{3/4}\varepsilon.
	\]
	This proves Bouland et al.'s conjecture on gently perturbed Gaussian permanents,
	Conjecture~6 in \cite{BoulandDattaFeffermanHernandez2025}, in the stronger
	form allowing arbitrary deterministic shifts; see their discussion for applications.  
	The anticoncentration estimate also refines the known asymptotic $\log|\per G_n|=(\tfrac12+o(1))n\log n$ \cite{TaoVu2009} to $O_{\Pp}(\log n)$ error:
	\begin{corollary}[Logarithmic asymptotics]
		\label{cor:logarithmic-size}
		For every $\mathbb K\in\{\R,\C,\Hh\}$, as $n\to\infty$,
		\[
		\log|\per_{\mathbb K}G_n^{\mathbb K}|
		=\frac12\log(n!)+O_{\Pp}(\log n)
		=\frac n2\log n-\frac n2+O_{\Pp}(\log n).
		\]
	\end{corollary}
	
	\begin{proof}
		Theorem~\ref{thm:pacc-headline} at $z=0$ and Chebyshev's inequality,
		using $\E|W_n^{\mathbb K}|^2=1$, give, for every $t\ge0$,
		\[
		\Pp\left(\left|\log|W_n^{\mathbb K}|\right|>t\right)
		\lesssim_\beta
		n^{(\beta+2)/4}e^{-\beta t}+e^{-2t}.
		\]
		Taking $t=M\log n$ for any $M>(\beta+2)/(4\beta)$ and applying
		Stirling's formula proves the claim.
	\end{proof}

	\subsection{Proof motivation via cofactor expansion.} To minimize notation, we explain the high-level structure of the proof in the commutative case $\mathbb K = \mathbb F \in \{\mathbb R, \mathbb C\}$ with center $z = 0$; the later sections give a rigorous proof of the full result.
	Anticoncentration of the permanent, in the sense of the PACC conjecture, is equivalent to a lower-tail bound on the magnitude of the permanent. In turn, our basic proof architecture can be motivated as a way to inductively prove a lower-tail estimate on $|\per G_n|^2$ via Chernoff bounds.
	
	Recall that the standard Chernoff bound shows that 
	\[ \Pp(|\per G_n|^2\le t) \le \inf_{s > 0} e^{st}\E e^{-s|\per G_n|^2}, \]
	so it suffices to prove a sufficiently strong upper bound on the Laplace transform $\E e^{-s|\per G_n|^2}$. To try to do this by induction, we can use the cofactor expansion of $\per G_n$ to rewrite the permanent of an $n \times n$ matrix as a random linear combination of permanents of $(n - 1) \times (n - 1)$ minors.
	The dependence among $(n - 1) \times (n - 1)$ permanents of
	overlapping minors is the main obstacle to a direct anticoncentration proof \cite{TaoMathOverflow2010}. 
	
	For the Gaussian case, the proof rests on comparing permanental and determinantal cofactors in Laplace-transform order.  The determinantal cofactors are tractable because the signed cofactor vector is orthogonal to each row of the matrix: its inner product with any vector $u$ equals $\det(r_1;\ldots;r_{n-1};u)$ where $r_1,\ldots,r_{n-1}$ are the rows, so taking $u$ equal to any existing row gives a matrix with a repeated row, whose determinant vanishes.  Orthogonality to all $n-1$ rows confines the cofactor vector to the 1-dimensional complement of the row span, which is uniformly oriented by Gaussian rotational invariance; its law therefore satisfies an exact closed-form recursion.  The comparison, Theorem~\ref{thm:determinantal-cofactor-comparison} below, transfers this tractability to the permanent.  To state it precisely we first recall two definitions:

	\begin{definition}[Scalar Laplace-transform order]
		Following the standard convention
		\cite{AlzaidKimProschan1991,ShakedShanthikumar2007}, for nonnegative random
		variables $U,V$ write $U\le_{\mathrm{Lt}}V$ if
		\[
		\E e^{-sU}\ge\E e^{-sV}\qquad(s\ge0).
		\]
	\end{definition}
	
	\begin{definition}[Matrix-variate Laplace-transform order]
		This order was introduced for real symmetric matrices in
		\cite[Definition~2.13]{GenestOuimetRichards2024}.  For random positive
		semidefinite Hermitian $\mathbb K$-valued matrices $M,N$ of the same size,
		write
		$M\preceq_{\mathrm{Lt}}N$ if
		\begin{equation}\label{eq:matrix-laplace-order}
			\E e^{-\RePart\operatorname{tr}(TM)}
			\ge \E e^{-\RePart\operatorname{tr}(TN)}
			\qquad(T\succeq0).
		\end{equation}
	\end{definition}
	
	Let $B_n$ be an $(n-1)\times n$
	standard $\F$-Gaussian matrix and
	define its permanental and signed determinantal cofactor\footnote{The Study determinant is not defined via cofactor expansion, which is why we have restricted to $\F \in \{\R,\C\}$ for this explanation. See the proof for the details of the quaternion case.}  column vectors by
	\[
	(C_n^{\mathrm{per}})_j=\per((B_n)_{-j}),
	\qquad
	(C_n^{\det})_j=(-1)^{j-1}\det((B_n)_{-j}),
	\qquad 1\le j\le n,
	\]
	where $(B_n)_{-j}$ denotes deletion of column $j$, letting
	$C_1^{\mathrm{per}}=C_1^{\det}=(1)$. The key fact powering the Gaussian induction is the following comparison between the permanent and the determinant:
	\begin{theorem}[Determinantal cofactor comparison]
		\label{thm:determinantal-cofactor-comparison}
		For each $\F\in\{\R,\C\}$ and every $n\ge1$,
		\begin{equation}\label{eq:cofactor-matrix-comparison}
			C_n^{\det}(C_n^{\det})^*
			\preceq_{\mathrm{Lt}}
			C_n^{\mathrm{per}}(C_n^{\mathrm{per}})^*.
		\end{equation}
		Moreover,
		\begin{equation}\label{eq:determinant-permanent-comparison}
			|\det G_n|^2\le_{\mathrm{Lt}}|\per G_n|^2.
		\end{equation}
		Finally, for every $z\in\F^{n+1}$,
		\begin{equation}\label{eq:cofactor-fourier-comparison}
			0\le
			\E e^{\ii\RePart\langle z,C_{n+1}^{\mathrm{per}}\rangle}
			\le
			\E e^{\ii\RePart\langle z,C_{n+1}^{\det}\rangle}.
		\end{equation}
	\end{theorem}
	The main technical step is a \emph{compression} argument---a one-sided replacement for the determinant's spherical invariance, showing that no test direction concentrates the cofactor vector more than a single coordinate---underlying the last Fourier inequality above; we leave the details to the proof.
	The permanent induction then runs the same cycle as the determinant's, with the exact rotational identities replaced by these one-sided comparisons.

	In the actual proof, the determinant enters implicitly: the inductive benchmark is a product of independent gamma variables whose law coincides with that of $|\det G_n|^2$.
	Section~\ref{sec:determinantal-benchmark} explains the connection between this gamma benchmark and
	determinant and Study determinant models.  Section~\ref{sec:gaussian-core}
	proves the common Gaussian cofactor induction, its comparison consequences,
	and Theorem~\ref{thm:pacc-headline}.
	Section~\ref{sec:gaussian-perturbations} proves stability under arbitrary
	independent additive perturbations, concluding the main body of the paper.
	
	Appendix~\ref{app:stronger-orders} shows
	that the Laplace-transform comparison generally cannot be strengthened to
	stochastic or convex order.  
	Appendix~\ref{app:cayley-determinant} studies the
	quaternionic Cayley (i.e., row-ordered) determinant obtained by inserting
	$\operatorname{sgn}(\sigma)$ into the permanent sum.  Although this Cayley
	determinant is different from the Study determinant, the same cofactor argument
	proves
	$\Sdet(G_n^{\Hh})^2\le_{\mathrm{Lt}}|\Cdet G_n^{\Hh}|^2$; low-dimensional
	moments distinguish the Study determinant, Cayley determinant, and permanent.
	Finally, Appendix~\ref{app:independent-rows} gives sufficient projection
	conditions under which independent, possibly non-Gaussian rows retain the
	Ginibre comparison at the exact second-moment scale.
	
	\subsection{Notation}
	When $\mathbb K=\mathbb F\in\{\R,\C\}$, we suppress field superscripts
	when no confusion can arise.
	Regard
	elements of $\mathbb K^m$ as column vectors and write
	\[
	\langle z,w\rangle=\sum_{j=1}^m\overline{z_j}w_j,
	\qquad
	(z,w)_{\R}=\RePart\langle z,w\rangle.
	\]
	Note that $(\cdot,\cdot)_{\R}$ is the ordinary Euclidean inner product on the
	underlying real space $\mathbb{R}^{\beta m}$.  The real Euclidean
	structure is what will be used for all Fourier transforms and Gaussian integrations below.
	
	Set
	\[
	X_n^{\mathbb K}
	:=\frac{|\per_{\mathbb K}G_n^{\mathbb K}|^2}{n!},
	\qquad
	\Delta_n^{\mathbb K}
	:=
	\begin{cases}
		|\det G_n^{\mathbb K}|^2/n!,&\mathbb K\in\{\R,\C\},\\
		\Sdet(G_n^{\Hh})^2/n!,&\mathbb K=\Hh.
	\end{cases}
	\]
	Here $\Sdet A=\prod_j\sigma_j(A)$ is the Study determinant, where the
	$\sigma_j(A)$ are the quaternionic singular values of $A$; see
	Definition~\ref{def:study-determinant}.  
	These normalizations are chosen so that
	\[
	\E X_n^{\mathbb K}=\E\Delta_n^{\mathbb K}=1.
	\]

	\section{Determinantal interpretation of the gamma benchmark}
	\label{sec:determinantal-benchmark}
	
	The gamma products in the Gaussian cofactor induction have concrete
	determinantal interpretations.  Over $\R$ and $\C$, they govern the exact laws
	of determinant cofactor vectors and their norms; over $\Hh$, the full product
	is the law of the squared Study determinant.  The real and complex determinant
	induction is the equality model for the permanent argument, although its
	identities are not formally used to prove the gamma comparisons.
	
	For each $\beta\in\{1,2,4\}$, let
	\[
	\gamma_{k,\beta}\sim
	\operatorname{Gamma}\left(\frac{\beta k}{2},\frac2\beta\right)
	\quad\text{independently over }k,
	\]
	where the parameters are shape and scale.  Thus $\gamma_{k,1}$ is
	chi-squared with $k$ degrees of freedom and
	$\gamma_{k,2}\sim\operatorname{Gamma}(k,1)$, while
	$\gamma_{k,4}\sim\operatorname{Gamma}(2k,1/2)$.
	For $m\ge0$, set
	\[
	\Pi_{m,\beta}=\prod_{k=2}^m\gamma_{k,\beta},
	\]
	with the empty product interpreted as $1$.  Thus
	$\Pi_{0,\beta}=\Pi_{1,\beta}=1$.
	
	In the real and complex cases and dimensions $n\ge2$, the determinant realizes
	these products geometrically: its cofactor vector has the law of a standard
	$\F$-Gaussian vector multiplied by the independent scale
	$\sqrt{\Pi_{n-1,\beta}}$, and its squared norm has law $\Pi_{n,\beta}$.
	
	\begin{proposition}[Gaussian determinant benchmark]
		\label{prop:determinant-benchmark}
		Let $\F\in\{\R,\C\}$ and $\beta=\dim_{\R}\F$.  For $n\ge2$, let $g_n$ be a
		standard $\F$-Gaussian column independent of $\Pi_{n-1,\beta}$.  Then
		\begin{equation}\label{eq:determinantal-gamma-vector}
			C_n^{\det}\ \stackrel{\mathrm d}=\
			\sqrt{\Pi_{n-1,\beta}}\,g_n.
		\end{equation}
		Consequently,
		\begin{align}
			C_n^{\det}(C_n^{\det})^*
			&\stackrel{\mathrm d}=
			\Pi_{n-1,\beta}g_ng_n^*,                              \label{eq:det-cofactor-matrix-law}\\
			\|C_n^{\det}\|_2^2
			&\stackrel{\mathrm d}=\Pi_{n,\beta}.
			\label{eq:det-cofactor-product}
		\end{align}
		Moreover, for every $n\ge1$,
		\begin{equation}\label{eq:det-gamma-intro}
			\Delta_n\ \stackrel{\mathrm d}=\
			\prod_{k=1}^n\frac{\gamma_{k,\beta}}{k}.
		\end{equation}
	\end{proposition}
	
	\begin{proof}
		For this proof, write $C(B)$ for the signed cofactor column of an
		$(n-1)\times n$ matrix $B$.  Expansion along the first row gives
		\[
		\det\begin{pmatrix}x\\ B\end{pmatrix}=xC(B)
		\qquad(x\in\F^{1\times n}).
		\]
		Let $Q\in\operatorname{SO}(n)$ when $\F=\R$, and
		$Q\in\operatorname{SU}(n)$ when $\F=\C$.  Since $\det Q=1$, the preceding
		identity yields
		\[
		xC(BQ)
		=\det\begin{pmatrix}x\\ BQ\end{pmatrix}
		=\det\left(\begin{pmatrix}xQ^{-1}\\ B\end{pmatrix}Q\right)
		=xQ^{-1}C(B).
		\]
		Hence $C(BQ)=Q^{-1}C(B)$.  Since $B_nQ\stackrel{\mathrm d}=B_n$, the law of
		$C_n^{\det}$ is invariant under $\operatorname{SO}(n)$ or
		$\operatorname{SU}(n)$.  Both groups act transitively on the unit sphere, so
		$C_n^{\det}$ is spherically invariant.
		
		We now prove inductively that
		\[
		\|C_n^{\det}\|_2^2\stackrel{\mathrm d}=\Pi_{n,\beta}
		\quad(n\ge1),
		\qquad
		C_{n+1}^{\det}\stackrel{\mathrm d}=
		\sqrt{\Pi_{n,\beta}}\,g_{n+1}
		\quad(n\ge1).
		\]
		The norm identity at $n=1$ is immediate from
		$C_1^{\det}=(1)$ and $\Pi_{1,\beta}=1$.  Suppose the norm identity holds at
		dimension $n$.  The first coordinate of $C_{n+1}^{\det}$ has the law of
		$\det G_n$.  Expanding this determinant along a standard Gaussian row $r$,
		independent of $B_n$, gives
		\[
		(C_{n+1}^{\det})_1
		\stackrel{\mathrm d}=rC_n^{\det}
		\stackrel{\mathrm d}=\|C_n^{\det}\|_2 g_1
		\stackrel{\mathrm d}=\sqrt{\Pi_{n,\beta}}\,g_1,
		\]
		where $g_1$ is a standard scalar $\F$-Gaussian independent of
		$\Pi_{n,\beta}$.
		
		The vector $\sqrt{\Pi_{n,\beta}}\,g_{n+1}$ is spherically invariant, as is
		$C_{n+1}^{\det}$ by the covariance proved above.  Spherically invariant
		vectors with the same first-coordinate law have the same distribution:
		rotational invariance identifies every real one-dimensional projection with
		a scaled first-coordinate projection, so this follows from the
		Cram\'er--Wold theorem.  Hence
		\[
		C_{n+1}^{\det}\stackrel{\mathrm d}=
		\sqrt{\Pi_{n,\beta}}\,g_{n+1}.
		\]
		Taking squared norms gives
		\[
		\|C_{n+1}^{\det}\|_2^2
		\stackrel{\mathrm d}=
		\Pi_{n,\beta}\|g_{n+1}\|_2^2
		\stackrel{\mathrm d}=
		\Pi_{n,\beta}\gamma_{n+1,\beta}
		=\Pi_{n+1,\beta},
		\]
		closing the induction.
		
		The vector identity in the induction, with the index shifted by one, proves
		\eqref{eq:determinantal-gamma-vector}.  Taking the outer product and squared
		norm gives
		\eqref{eq:det-cofactor-matrix-law} and
		\eqref{eq:det-cofactor-product}.  Finally,
		$(C_{n+1}^{\det})_1\stackrel{\mathrm d}=\det G_n$, so
		\[
		|\det G_n|^2
		\stackrel{\mathrm d}=
		\Pi_{n,\beta}|g_1|^2
		\stackrel{\mathrm d}=
		\prod_{k=1}^n\gamma_{k,\beta}.
		\]
		Dividing by $n!=\prod_{k=1}^n k$ proves
		\eqref{eq:det-gamma-intro}.
	\end{proof}
	
	The non-Hermitian Gaussian ensemble over $\Hh$ is usually called the
	\emph{symplectic Ginibre ensemble}, or GinSE; it is also called the quaternionic
	or quaternion-real Ginibre ensemble \cite{Ginibre1965,Forrester2010}.  Here
	\emph{symplectic} refers to its Dyson symmetry class, not to a symplectic
	constraint on the matrices.
	
	\begin{definition}[Study determinant]
		\label{def:study-determinant}
		Every $A\in\Hh^{n\times n}$ has a quaternionic singular-value decomposition
		\[
		A=U\operatorname{diag}(\sigma_1(A),\ldots,\sigma_n(A))V^*,
		\]
		where $U,V\in\Hh^{n\times n}$ are unitary and the singular values
		$\sigma_j(A)$ are nonnegative real numbers.  Equivalently,
		$\sigma_1(A)^2,\ldots,\sigma_n(A)^2$ are the eigenvalues of $A^*A$.  Define
		\[
		\Sdet A=\prod_{j=1}^n\sigma_j(A).
		\]
	\end{definition}
	
	This agrees with the Study determinant of Cohen and De Leo
	\cite[Corollary~6.5]{CohenDeLeo2000}.  It is multiplicative and agrees with
	$|\det A|$ when $A$ is complex
	\cite[Theorem~5.1 and Corollary~6.3]{CohenDeLeo2000}.  When $A$ is regarded as
	a real-linear operator,
	\[
	\det\nolimits_{\R}(A)=\Sdet(A)^4.
	\]
	For positive quaternionic Hermitian $H$, $\Sdet H$ is the product of the real
	eigenvalues and hence agrees with the Moore determinant
	\cite{Aslaksen1996}.  Consequently, $\Sdet(A)^2$ is the Moore determinant of
	$A^*A$.
	
	\begin{proposition}[Study determinant benchmark]
		\label{prop:study-determinant-benchmark}
		For every $n\ge1$,
		\begin{align}
			\Sdet(G_n^{\Hh})^2
			&\stackrel{\mathrm d}=\prod_{k=1}^n\gamma_{k,4},
			\label{eq:study-gamma-product}\\
			\Delta_n^{\Hh}
			=\frac{\Sdet(G_n^{\Hh})^2}{n!}
			&\stackrel{\mathrm d}=
			\prod_{k=1}^n\frac{\gamma_{k,4}}k.
			\label{eq:quaternion-determinant-product}
		\end{align}
	\end{proposition}
	
	\begin{proof}
		Apply Gram--Schmidt to the columns $v_1,\ldots,v_n$ of $G_n^{\Hh}$.
		This writes $G_n^{\Hh}=QR$, where $Q$ is unitary, $R=(r_{ij})$ is upper
		triangular, and
		\[
		r_{jj}=\operatorname{dist}\left(
		v_j,\operatorname{span}_{\Hh}(v_1,\ldots,v_{j-1})
		\right)\ge0.
		\]
		Multiplicativity and the triangular formula for the Study determinant
		\cite[Corollary~6.1]{CohenDeLeo2000} give
		\[
		\Sdet(G_n^{\Hh})^2=\prod_{j=1}^n r_{jj}^2.
		\]
		Conditionally on $v_1,\ldots,v_{j-1}$, the orthogonal complement of their
		span has quaternionic dimension $n-j+1$.  Unitary invariance therefore gives
		\[
		r_{jj}^2\stackrel{\mathrm d}=\gamma_{n-j+1,4},
		\qquad 1\le j\le n.
		\]
		This conditional law does not depend on the preceding columns, so the
		diagonal squares are independent.  Consequently,
		\[
		\Sdet(G_n^{\Hh})^2
		\stackrel{\mathrm d}=\prod_{k=1}^n\gamma_{k,4}.
		\]
		This proves
		\eqref{eq:study-gamma-product}, and division by $n!$ proves
		\eqref{eq:quaternion-determinant-product}.
	\end{proof}
	
	Thus the determinant and Study determinant realize the same scalar gamma
	benchmark in the three fields.
	
	\section{Gaussian cofactor induction}\label{sec:gaussian-core}
	
	We now prove the corresponding one-sided gamma comparisons directly from the
	products $\Pi_{n,\beta}$ and independent Gaussian vectors.
	For determinant cofactors, spherical invariance promotes the one-coordinate
	Gaussian recursion to an exact vector identity.  For permanental cofactors,
	Fourier coordinate compression supplies the corresponding one-sided step.
	
	\subsection{Real comparison lemmas}
	The following real comparison lemmas form the core of the proof.  In
	applications over $\mathbb K$, we identify $\mathbb K^m$ with $\R^{\beta m}$;
	the matrices below may therefore represent maps that are only real-linear.
	
	After conditioning on all but two Gaussian columns, multilinearity leaves a
	bilinear interaction between those columns together with one linear term in
	each.  We begin by analyzing characteristic functions of precisely this form.
	For $d,k\ge1$, $T\in\R^{d\times d}$, $L\in\R^{d\times k}$, and
	$a,b\in\R^k$, set
	\[
	F_{T,L}(a,b)=\E_{X,Y}\exp\left(\ii\bigl(
	X^*TY+X^*Lb+Y^*La\bigr)\right),
	\]
	where $X,Y$ are independent standard Gaussian column vectors in $\R^d$.
	
	\begin{lemma}[One-sided quadratic scaling]
		\label{lem:one-sided-quadratic-scaling}
		The quantities $F_{T,L}(a,0)$ and $F_{T,L}(0,b)$ are strictly positive.
		Moreover, for every $t\in\R$,
		\begin{align*}
			\frac{F_{T,L}(ta,0)}{F_{T,L}(0,0)}
			&=\left(\frac{F_{T,L}(a,0)}{F_{T,L}(0,0)}\right)^{t^2},\\
			\frac{F_{T,L}(0,tb)}{F_{T,L}(0,0)}
			&=\left(\frac{F_{T,L}(0,b)}{F_{T,L}(0,0)}\right)^{t^2}.
		\end{align*}
	\end{lemma}
	
	\begin{proof}
		Averaging first over $X$ in the first expression and over $Y$ in the second
		gives
		\begin{align*}
			F_{T,L}(a,0)
			&=\E_Y\exp\left(-\frac12\lVert TY\rVert^2+\ii Y^*La\right),\\
			F_{T,L}(0,b)
			&=\E_X\exp\left(-\frac12\lVert T^*X\rVert^2+\ii X^*Lb\right).
		\end{align*}
		For a standard Gaussian vector $Z$, a positive semidefinite matrix $Q$, and
		a real vector $h$, the centered Gaussian formula is
		\[
		\E\exp\left(-\frac12Z^*QZ+\ii h^*Z\right)
		=\det(I+Q)^{-1/2}
		\exp\left(-\frac12h^*(I+Q)^{-1}h\right).
		\]
		Indeed, the quadratic weight changes the standard Gaussian density into
		$\det(I+Q)^{-1/2}$ times the density of a centered Gaussian with covariance
		$(I+Q)^{-1}$.  Applying this formula to the preceding expectations and using
		$\det(I+T^*T)=\det(I+TT^*)$ gives
		\[
		F_{T,L}(0,0)
		=\det(I+T^*T)^{-1/2}
		=\det(I+TT^*)^{-1/2}>0.
		\]
		Dividing by this common factor gives
		\begin{align*}
			\frac{F_{T,L}(a,0)}{F_{T,L}(0,0)}
			&=\exp\left(-\frac12a^*L^*(I+T^*T)^{-1}La\right),\\
			\frac{F_{T,L}(0,b)}{F_{T,L}(0,0)}
			&=\exp\left(-\frac12b^*L^*(I+TT^*)^{-1}Lb\right).
		\end{align*}
		These formulas prove strict positivity, and replacing $a,b$ by $ta,tb$ gives
		the two scaling identities.
	\end{proof}
	
	The key interpolation uses orthogonal invariance in $\R^{2d}$.  Its
	coefficients can be parameterized by
	\[
	(\sqrt\theta,\sqrt{1-\theta})=(\cos\alpha,\sin\alpha),
	\qquad 0\le\alpha\le\frac\pi2,
	\]
	which places them on the unit quarter-circle.
	
	\begin{lemma}[Gaussian geometric interpolation]
		\label{lem:gaussian-geometric-interpolation}
		For every $T,L,a,b$ as above and $\theta\in[0,1]$,
		\[
		\left|F_{T,L}\left(
		\sqrt\theta\,a,\sqrt{1-\theta}\,b\right)\right|
		=F_{T,L}(a,0)^\theta F_{T,L}(0,b)^{1-\theta}.
		\]
	\end{lemma}
	
	\begin{proof}
		Fix $T,L$.
		We first prove the equal-weight ($\theta = 1/2$) identity
		\[
		\left|F_{T,L}\left(\frac a{\sqrt2},\frac b{\sqrt2}\right)\right|^2
		=F_{T,L}(a,0)F_{T,L}(0,b).
		\]
		If $U'$ is an independent copy of a complex random variable $U$, then
		$|\E U|^2=\E[U\overline{U'}]$.  Apply this with
		\[
		U=\exp\left(\ii(X^*TY+X^*Lb/\sqrt2+Y^*La/\sqrt2)\right).
		\]
		Let $(X',Y')$ be an independent copy of $(X,Y)$, and set
		\[
		X_\pm=\frac{X\pm X'}{\sqrt2},
		\qquad
		Y_\pm=\frac{Y\pm Y'}{\sqrt2}.
		\]
		Applied separately to $(X,X')$ and $(Y,Y')$, this is an orthogonal
		transformation of $\R^{2d}$.  Thus $X_+,X_-,Y_+,Y_-$ are independent standard
		Gaussian vectors, and under this change of variables,
		\begin{align*}
			X^*TY-X'^*TY'&=X_+^*TY_-+X_-^*TY_+,\\
			\frac{(X-X')^*Lb}{\sqrt2}&=X_-^*Lb,\\
			\frac{(Y-Y')^*La}{\sqrt2}&=Y_-^*La.
		\end{align*}
		Combining these formulas shows
		\[
		U\overline{U'} = \exp(\ii[(X_+^*TY_-+Y_-^*La)+(X_-^*TY_++X_-^*Lb)]).
		\]
		The two summands depend on the independent pairs $(X_+,Y_-)$ and
		$(X_-,Y_+)$, respectively.  Therefore the expectation factors, giving
		\[
		\left|F_{T,L}\left(\frac a{\sqrt2},\frac b{\sqrt2}\right)\right|^2
		=F_{T,L}(a,0)F_{T,L}(0,b),
		\]
		as claimed.
		
		Apply the equal-weight identity with
		$a$ replaced by $\sqrt{2\theta}\,a$ and $b$ by
		$\sqrt{2(1-\theta)}\,b$, and then use
		Lemma~\ref{lem:one-sided-quadratic-scaling}.  This gives
		\begin{align*}
			\left|F_{T,L}\left(\sqrt\theta\,a,
			\sqrt{1-\theta}\,b\right)\right|^2
			&=F_{T,L}(\sqrt{2\theta}\,a,0)
			F_{T,L}(0,\sqrt{2(1-\theta)}\,b)\\
			&=F_{T,L}(0,0)^2
			\left(\frac{F_{T,L}(a,0)}{F_{T,L}(0,0)}\right)^{2\theta}
			\left(\frac{F_{T,L}(0,b)}{F_{T,L}(0,0)}\right)^{2(1-\theta)}\\
			&=F_{T,L}(a,0)^{2\theta}
			F_{T,L}(0,b)^{2(1-\theta)}.
		\end{align*}
		The last equality uses $2\theta+2(1-\theta)=2$.  Taking square roots proves
		the result.
	\end{proof}
	
	\begin{corollary}[Gaussian two-coordinate compression]
		\label{cor:gaussian-two-coordinate-compression}
		Let $T\in\R^{d\times d}$ and $L\in\R^{d\times k}$ be jointly random
		matrices.  For every $a,b\in\R^k$ and $\theta\in[0,1]$,
		\begin{align*}
			&\E_{T,L}\left|F_{T,L}\left(
			\sqrt\theta\,a,\sqrt{1-\theta}\,b\right)\right|\\
			&\qquad\le
			\bigl(\E_{T,L}F_{T,L}(a,0)\bigr)^\theta
			\bigl(\E_{T,L}F_{T,L}(0,b)\bigr)^{1-\theta}\\
			&\qquad\le
			\max\left\{
			\E_{T,L}F_{T,L}(a,0),
			\E_{T,L}F_{T,L}(0,b)
			\right\}.
		\end{align*}
	\end{corollary}
	
	\begin{proof}
		The first inequality follows from
		Lemma~\ref{lem:gaussian-geometric-interpolation} and H\"older's inequality;
		the endpoint cases are immediate.  The second says that a weighted geometric
		mean is at most its larger endpoint.
	\end{proof}
	
	The other real input turns Fourier domination of random vectors into matrix
	Laplace domination of their outer products.
	
	\begin{lemma}[Fourier-to-Laplace comparison]
		\label{lem:fourier-outer-products}
		Let $X,Y$ be random column vectors in $\R^d$.  Suppose
		\[
		\RePart\E e^{\ii t^*X}
		\le \RePart\E e^{\ii t^*Y}
		\qquad(t\in\R^d).
		\]
		Then
		\[
		YY^*\preceq_{\mathrm{Lt}}XX^*.
		\]
	\end{lemma}
	
	\begin{proof}
		Fix a positive definite real symmetric matrix $Q$.  Fourier inversion gives
		\[
		e^{-x^*Qx}
		=\frac{1}{(4\pi)^{d/2}\sqrt{\det Q}}
		\int_{\R^d}
		e^{-t^*Q^{-1}t/4}e^{\ii t^*x}\dd t.
		\]
		The Fourier kernel is nonnegative and integrable.  Averaging at $x=X$, taking
		real parts, and applying the assumed comparison therefore gives
		\[
		\E e^{-X^*QX}\le \E e^{-Y^*QY}.
		\]
		Since $X^*QX=\operatorname{tr}(QXX^*)$, this is the defining Laplace
		comparison for positive definite tests.  Replacing a positive semidefinite
		$Q$ by $Q+\varepsilon I$ and letting $\varepsilon\downarrow0$ completes the
		proof.
	\end{proof}
	
	\subsection{Fourier compression for permanental cofactors}
	
	Fix $\mathbb K\in\{\R,\C,\Hh\}$, put $\beta=\dim_{\R}\mathbb K$, and let
	$n\ge2$.  Let $B$ be an $(n-1)\times n$ standard $\mathbb K$-Gaussian
	matrix and define its permanental cofactor column by
	\[
	(C_n^{\mathbb K})_j=\per_{\mathbb K}(B_{-j}),
	\qquad 1\le j\le n,
	\qquad
	C_1^{\mathbb K}=(1).
	\]
	Column-permutation invariance of $\per_{\mathbb K}$ makes the coordinates of
	$C_n^{\mathbb K}$ exchangeable.
	When $n$ and $\mathbb K$ are fixed, write $C=C_n^{\mathbb K}$.
	
	\begin{proposition}[Fourier coordinate compression]\label{prop:compression}
		Let $n \ge 2$ and define
		\[
		\varphi_n^{\mathbb K}(z)
		:=\E e^{\ii(z,C)_{\R}}
		=\E\exp\left(\ii\RePart
		\sum_{j=1}^n\overline{z_j}C_j\right).
		\]
		For every $z\in\mathbb K^n$,
		\begin{equation}\label{eq:full-radial-domination}
			0\le\varphi_n^{\mathbb K}(z)
			\le\varphi_n^{\mathbb K}(\|z\|_2e_1).
		\end{equation}
	\end{proposition}
	
	\begin{proof}
		Put $w=\overline z$.  The sum $\sum_jw_jC_j$ is $\per_{\mathbb K}$ of the
		matrix obtained by placing $w$ above $B$.  Since $n\ge2$, the matrix $B$ has
		at least one Gaussian row.  Condition on all rows except its last row
		$h=(h_1,\ldots,h_n)$.  The permanent has the form
		\[
		\sum_{j=1}^nD_jh_j
		\]
		for conditional coefficients $D_j\in\mathbb K$.\footnote{For
			$\mathbb K=\Hh$, left and right multiplication by a fixed quaternion $q$ are
			scaled orthogonal maps of $\Hh\simeq\R^4$, since
			$|qx|=|xq|=|q||x|$.}  Conditional on the other rows, the $D_j$ are fixed.
		For fixed $D_j$, the real random variable $\RePart(D_jh_j)$ is centered
		Gaussian with variance $|D_j|^2/\beta$: multiplication by $D_j$ is a scaled
		orthogonal map on the underlying real space.  These variables are independent
		over $j$, and hence
		\[
		\E_h\exp\left(\ii\RePart\sum_{j=1}^nD_jh_j\right)
		=\exp\left(-\frac1{2\beta}\sum_{j=1}^n|D_j|^2\right)\ge0.
		\]
		Averaging proves the lower bound in
		\eqref{eq:full-radial-domination}.  The restriction $n\ge2$ is genuine here:
		at $n=1$ there is no Gaussian row to average over, and $C_1^{\mathbb K}=(1)$
		has characteristic function $e^{\ii\RePart z}$, which need not be nonnegative.
		
		For the upper bound, if $z$ has at most one nonzero coordinate, column
		exchangeability and unit-scalar invariance give the result
		directly.\footnote{For $\mathbb K=\Hh$, left-multiplying the first row of
			$B$ by a unit quaternion preserves its law and left-multiplies every
			coordinate of $C_n^{\mathbb K}$ by that quaternion.}
		Otherwise choose two nonzero coordinates and relabel them as the first two.
		Write the first two columns of $B$ as $X,Y\in\mathbb K^{n-1}$ and denote the
		remaining columns, indexed by $3,\ldots,n$, by $R$.  Expanding the permanent
		along the first row $w$ gives
		\begin{align}
			\sum_{j=1}^n w_jC_j
			&=w_1\per_{\mathbb K}[Y,R]+w_2\per_{\mathbb K}[X,R]
			+\sum_{j=3}^n w_j\per_{\mathbb K}[X,Y,R_{-j}].
			\label{eq:two-column-expansion}
		\end{align}
		Identify $\mathbb K^{n-1}$ with $\R^d$, where $d=\beta(n-1)$, and
		$\mathbb K$ with $\R^\beta$.  For fixed $R$, let
		$T_R\in\R^{d\times d}$ and $L_R\in\R^{d\times\beta}$ be characterized, for
		$U,V\in\mathbb K^{n-1}$ and $a\in\mathbb K$, by
		\begin{align*}
			(U,T_RV)_{\R}
			&=\RePart\sum_{j=3}^nw_j\per_{\mathbb K}[U,V,R_{-j}],\\
			(U,L_Ra)_{\R}
			&=\RePart\bigl(a\per_{\mathbb K}[U,R]\bigr).
		\end{align*}
		The real part of \eqref{eq:two-column-expansion} is
		\[
		(X,T_RY)_{\R}+(X,L_Rw_2)_{\R}+(Y,L_Rw_1)_{\R}.
		\]
		
		The real covariance of a standard $\mathbb K$-Gaussian is $I/\beta$.  After
		rescaling $X,Y$ by $\sqrt\beta$, the conditional expectation over $X,Y$ is
		$F_{T_R/\beta,L_R/\sqrt\beta}(w_1,w_2)$.
		
		Let
		$r=\sqrt{|z_1|^2+|z_2|^2}=\sqrt{|w_1|^2+|w_2|^2}$.  Apply
		Corollary~\ref{cor:gaussian-two-coordinate-compression} to the random pair
		$(T_R/\beta,L_R/\sqrt\beta)$, with endpoint vectors $rw_1/|w_1|$ and
		$rw_2/|w_2|$ and interpolation parameter $|w_1|^2/r^2$.  The conditional
		representation above gives $\varphi_n^{\mathbb K}(z)=\E_R F_{T_R/\beta,L_R/\sqrt\beta}(w_1,w_2)$,
		which is nonneg\-ative by the lower bound already proved.  Therefore,
		\[
		\varphi_n^{\mathbb K}(z)
		=\left|\E_R F_{T_R/\beta,L_R/\sqrt\beta}(w_1,w_2)\right|
		\le\E_R\left|F_{T_R/\beta,L_R/\sqrt\beta}(w_1,w_2)\right|.
		\]
		The corollary bounds the right side by the larger of its two endpoint
		expectations.  The endpoint vectors $rw_j/|w_j|$ correspond to arguments
		$rz_j/|z_j|$ in $\varphi_n^{\mathbb K}$ since $w=\overline z$.
		These are precisely the characteristic-function values in
		\[
		\varphi_n^{\mathbb K}(z)
		\le
		\max\left\{
		\varphi_n^{\mathbb K}
		\left(r\frac{z_1}{|z_1|},0,z_3,\ldots,z_n\right),
		\varphi_n^{\mathbb K}
		\left(0,r\frac{z_2}{|z_2|},z_3,\ldots,z_n\right)
		\right\}.
		\]
		Thus one of these two replacements does not decrease the characteristic
		function and reduces the number of nonzero coordinates by one.  Iterating
		leaves one coordinate of modulus $\|z\|_2$.  Column exchangeability and the
		same unit-scalar invariance identify the resulting value with
		$\varphi_n^{\mathbb K}(\|z\|_2e_1)$.
	\end{proof}
	
	\subsection{The gamma form of the cofactor induction}
	
	\begin{theorem}[Gaussian cofactor induction]
		\label{thm:gaussian-cofactor-gamma}
		Let $\mathbb K\in\{\R,\C,\Hh\}$ and $\beta=\dim_{\R}\mathbb K$.
		For each $m\ge1$, let $g_m^{\mathbb K}$ be a standard
		$\mathbb K$-Gaussian column in $\mathbb K^m$, independent of the gamma
		variables.  Then, for every $n\ge1$,
		\begin{equation}\label{eq:permanental-gamma-matrix}
			\Pi_{n-1,\beta}g_n^{\mathbb K}(g_n^{\mathbb K})^*
			\preceq_{\mathrm{Lt}}
			C_n^{\mathbb K}(C_n^{\mathbb K})^*.
		\end{equation}
		Moreover,
		\begin{equation}\label{eq:cofactor-gamma-order}
			\Pi_{n,\beta}
			\le_{\mathrm{Lt}}\|C_n^{\mathbb K}\|_2^2.
		\end{equation}
		Finally, for every $z\in\mathbb K^{n+1}$,
		\begin{equation}\label{eq:permanental-gamma-fourier}
			0\le \E e^{\ii(z,C_{n+1}^{\mathbb K})_{\R}}
			\le \E e^{\ii(z,\sqrt{\Pi_{n,\beta}}\,
				g_{n+1}^{\mathbb K})_{\R}}.
		\end{equation}
	\end{theorem}
	
	\begin{proof}
		At $n=1$, \eqref{eq:cofactor-gamma-order} is the equality
		$\Pi_{1,\beta}=\|C_1^{\mathbb K}\|_2^2=1$.  The matrix comparison there is
		$|g_1^{\mathbb K}|^2\le_{\mathrm{Lt}}1$, which follows from Jensen's
		inequality:
		\[
		\E e^{-s|g_1^{\mathbb K}|^2}\ge
		e^{-s\E|g_1^{\mathbb K}|^2}=e^{-s}.
		\]
		Suppose that the scalar
		comparison holds at dimension $n$.  Proposition~\ref{prop:compression}
		gives
		\[
		0\le \E e^{\ii(z,C_{n+1}^{\mathbb K})_{\R}}
		\le
		\E e^{\ii(\|z\|_2e_1,C_{n+1}^{\mathbb K})_{\R}}.
		\]
		The first coordinate of $C_{n+1}^{\mathbb K}$ has the law of
		$\per_{\mathbb K}G_n^{\mathbb K}$.  Expanding along its first row and
		conditioning on the remaining rows gives, for $r\ge0$,
		\[
		\E e^{\ii(re_1,C_{n+1}^{\mathbb K})_{\R}}
		=\E e^{-r^2\|C_n^{\mathbb K}\|_2^2/(2\beta)}.
		\]
		The induction hypothesis therefore bounds the last expression by
		\[
		\E e^{-\Pi_{n,\beta}\|z\|_2^2/(2\beta)}
		=\E e^{\ii(z,\sqrt{\Pi_{n,\beta}}\,
			g_{n+1}^{\mathbb K})_{\R}},
		\]
		proving \eqref{eq:permanental-gamma-fourier}.
		
		Apply Lemma~\ref{lem:fourier-outer-products} on the underlying real space
		with $X=C_{n+1}^{\mathbb K}$ and
		$Y=\sqrt{\Pi_{n,\beta}}\,g_{n+1}^{\mathbb K}$.  The Fourier comparison just
		proved is exactly its hypothesis, so the lemma gives
		\eqref{eq:permanental-gamma-matrix} in dimension
		$n+1$.\footnote{For $\mathbb K=\Hh$, restricting
			the real matrix tests to quaternionic Hermitian $T\succeq0$ is valid because
			$x\mapsto\RePart(x^*Tx)$ is a positive semidefinite real quadratic form and
			$\RePart\operatorname{tr}(Txx^*)=\RePart(x^*Tx)$, using cyclicity of the real
			trace.}  Taking the trace test (i.e., specializing $T=sI$ in
		\eqref{eq:matrix-laplace-order}) in \eqref{eq:permanental-gamma-matrix} and
		using
		\[
		\Pi_{n,\beta}\|g_{n+1}^{\mathbb K}\|_2^2
		\stackrel{\mathrm d}=\Pi_{n+1,\beta}
		\]
		proves \eqref{eq:cofactor-gamma-order} at dimension $n+1$ and closes the
		induction.
	\end{proof}
	
	\subsection{Determinant and Study determinant comparisons}
	\label{subsec:comparison-consequences}
	
	\begin{proof}[Proof of Theorem~\ref{thm:determinantal-cofactor-comparison}]
		At $n=1$, the matrix and determinant--permanent comparisons are equalities.
		For $n\ge2$, specialize Theorem~\ref{thm:gaussian-cofactor-gamma} to
		$\mathbb K=\F$, so $C_n^{\mathbb K}=C_n^{\mathrm{per}}$.
		The matrix identity \eqref{eq:det-cofactor-matrix-law} in
		Proposition~\ref{prop:determinant-benchmark} identifies the benchmark in
		\eqref{eq:permanental-gamma-matrix} with
		$C_n^{\det}(C_n^{\det})^*$, proving
		\eqref{eq:cofactor-matrix-comparison} at dimension $n$.
		
		For every $n\ge1$, the vector identity
		\eqref{eq:determinantal-gamma-vector} at dimension $n+1$ identifies the
		Gaussian benchmark in \eqref{eq:permanental-gamma-fourier} with
		$C_{n+1}^{\det}$.  The Fourier conclusion there therefore proves
		\eqref{eq:cofactor-fourier-comparison}.
		
		Finally, fix $n\ge2$.  Taking scalar trace tests in
		\eqref{eq:cofactor-matrix-comparison} gives
		\[
		\|C_n^{\det}\|_2^2
		\le_{\mathrm{Lt}}
		\|C_n^{\mathrm{per}}\|_2^2.
		\]
		Expansion along an independent standard Gaussian row gives
		\[
		|\det G_n|^2\stackrel{\mathrm d}=
		\gamma_{1,\beta}\|C_n^{\det}\|_2^2,
		\qquad
		|\per G_n|^2\stackrel{\mathrm d}=
		\gamma_{1,\beta}\|C_n^{\mathrm{per}}\|_2^2,
		\]
		with independent factors.  Multiplication by the common nonnegative gamma
		factor preserves scalar Laplace-transform order, proving
		\eqref{eq:determinant-permanent-comparison} at dimension $n$.
	\end{proof}
	
	\begin{theorem}[Study determinant--permanent comparison]
		\label{thm:quaternion-laplace}
		For every $n\ge1$,
		\begin{equation}\label{eq:quaternion-laplace-comparison}
			\Sdet(G_n^{\Hh})^2
			\le_{\mathrm{Lt}}|\per_{\Hh}G_n^{\Hh}|^2.
		\end{equation}
		Equivalently, $\Delta_n^{\Hh}\le_{\mathrm{Lt}}X_n^{\Hh}$.
	\end{theorem}
	
	\begin{proof}
		Expansion along the first row gives
		\[
		|\per_{\Hh}G_n^{\Hh}|^2
		\stackrel{\mathrm d}
		=\gamma_{1,4}\|C_n^{\Hh}\|_2^2,
		\]
		with independent factors on the right.  Multiplying the $\mathbb K=\Hh$
		case of \eqref{eq:cofactor-gamma-order} by an independent $\gamma_{1,4}$
		therefore gives
		\[
		\prod_{k=1}^n\gamma_{k,4}
		\le_{\mathrm{Lt}}|\per_{\Hh}G_n^{\Hh}|^2.
		\]
		By \eqref{eq:study-gamma-product}, the left side has the law of
		$\Sdet(G_n^{\Hh})^2$, proving
		\eqref{eq:quaternion-laplace-comparison}.
	\end{proof}
	
	\subsection{Gaussian anticoncentration}
	
	Let $\omega_\beta=\pi^{\beta/2}/\Gamma(\beta/2+1)$ be the volume of the unit
	ball in $\R^\beta$.
	With respect to Lebesgue measure on
	$\mathbb K\simeq\R^\beta$, the standard $\mathbb K$-Gaussian density is
	\[
	\left(\frac{\beta}{2\pi}\right)^{\beta/2}e^{-\beta|z|^2/2}.
	\]
	Put
	\[
	D_{n,\beta}=\prod_{k=2}^n\frac{\gamma_{k,\beta}}k
	=\frac{\Pi_{n,\beta}}{n!},
	\]
	where an empty product equals $1$.  Since
	$\Delta_n^{\mathbb K}\stackrel{\mathrm d}
	=D_{n,\beta}\gamma_{1,\beta}$ with independent factors,
	\[
	\E e^{-s\Delta_n^{\mathbb K}}
	=\E\left(1+\frac{2sD_{n,\beta}}\beta\right)^{-\beta/2}.
	\]
	For each $D_{n,\beta}>0$,
	\[
	\left(\frac{s}{\pi}\right)^{\beta/2}
	\left(1+\frac{2sD_{n,\beta}}\beta\right)^{-\beta/2}
	\uparrow
	\left(\frac{\beta}{2\pi}\right)^{\beta/2}D_{n,\beta}^{-\beta/2}
	\qquad(s\to\infty).
	\]
	Monotone convergence and the gamma moment formula therefore give
	\begin{equation}\label{eq:kappa-definition}
		\kappa_{n,\beta}
		:=\lim_{s\to\infty}
		\left(\frac{s}{\pi}\right)^{\beta/2}
		\E e^{-s\Delta_n^{\mathbb K}}
		=\left(\frac{\beta}{2\pi}\right)^{\beta/2}
		\E D_{n,\beta}^{-\beta/2}
		=\frac{(n!)^{\beta/2}}{\pi^{\beta/2}}
		\left(\frac\beta2\right)^{\beta n/2}
		\frac{\Gamma(\beta/2)}{\Gamma(\beta n/2)}.
	\end{equation}
	In the three cases of interest,
	\begin{equation}\label{eq:kappa-special-cases}
		\kappa_{n,1}=\frac{\sqrt{n!}}{2^{n/2}\Gamma(n/2)},
		\qquad
		\kappa_{n,2}=\frac n\pi,
		\qquad
		\kappa_{n,4}=\frac{2n}{\pi^2}\frac{4^n(n!)^2}{(2n)!}.
	\end{equation}
	Stirling's formula yields
	\begin{equation}\label{eq:kappa-asymptotic}
		\kappa_{n,\beta}\asymp_\beta n^{(\beta+2)/4}.
	\end{equation}
	The scalar Laplace comparison already gives centered anticoncentration by the
	Chernoff argument from the introduction.  More generally, suppose that $Y\ge0$
	satisfies $\Delta_n^{\mathbb K}\le_{\mathrm{Lt}}Y$.  For every
	$\varepsilon>0$, applying Markov's inequality with $s=\beta/(2\varepsilon^2)$,
	then the Laplace-transform comparison and the formula computed above, then
	$1+x\ge x$ for $x=D_{n,\beta}/\varepsilon^2>0$,
	\begin{align}
		\Pp(Y\le\varepsilon^2)
		&\le e^{\beta/2}\E e^{-\beta Y/(2\varepsilon^2)}\notag\\
		&\le e^{\beta/2}\E\left(1+\frac{D_{n,\beta}}{\varepsilon^2}\right)^{-\beta/2}\notag\\
		&\le e^{\beta/2}\varepsilon^\beta\E D_{n,\beta}^{-\beta/2}
		=\left(\frac{2\pi e}{\beta}\right)^{\beta/2}
		\kappa_{n,\beta}\varepsilon^\beta.
		\label{eq:laplace-chernoff-small-ball}
	\end{align}
	Taking $Y=X_n^{\mathbb K}$ gives
	\begin{equation}\label{eq:ginibre-chernoff-small-ball}
		\Pp(|W_n^{\mathbb K}|\le\varepsilon)
		\le\min\left\{1,
		\left(\frac{2\pi e}{\beta}\right)^{\beta/2}
		\kappa_{n,\beta}\varepsilon^\beta\right\}.
	\end{equation}
	In the complex case this is $\min\{1,en\varepsilon^2\}$, already enough for
	PACC.  The density argument below is stronger: it is uniform over translated
	balls and, in the complex case, removes the factor $e$.
	
	\begin{proof}[Proof of Theorem~\ref{thm:pacc-headline}]
		Let $C_n^{\mathbb K}$ be the cofactor column determined by all but the first
		row, and let $g_1^{\mathbb K}$ be an independent standard scalar
		$\mathbb K$-Gaussian.  Expansion along the first row and Gaussian rotational
		invariance give the conditional identity
		\begin{equation}\label{eq:permanent-conditional-mixture}
			W_n^{\mathbb K}\mid C_n^{\mathbb K}
			\ \stackrel{\mathrm d}=\
			\frac{\|C_n^{\mathbb K}\|_2}{\sqrt{n!}}g_1^{\mathbb K}.
		\end{equation}
		Thus $W_n^{\mathbb K}$ is a mixture of centered radial Gaussian densities, so
		its density $p_n^{\mathbb K}$ is radial and nonincreasing in $|z|$.
		Conditioned on $C_n^{\mathbb K}$, the Laplace transform of $X_n^{\mathbb K}$
		is $\bigl(1+2sD_{n,\beta}/\beta\bigr)^{-\beta/2}$ with
		$D_{n,\beta}=\|C_n^{\mathbb K}\|_2^2/n!$, so
		$(s/\pi)^{\beta/2}\E[e^{-sX_n^{\mathbb K}}\mid C_n^{\mathbb K}]$ increases
		pointwise to $(\beta/(2\pi D_{n,\beta}))^{\beta/2}$, the conditional Gaussian
		density at $0$.  Monotone convergence gives the Gaussian approximate-identity
		limit
		\begin{equation}\label{eq:ginibre-density-laplace-limit}
			\left(\frac{s}{\pi}\right)^{\beta/2}
			\E e^{-sX_n^{\mathbb K}}
			\uparrow p_n^{\mathbb K}(0)
			\qquad(s\to\infty).
		\end{equation}
		The comparison theorems above give
		$\Delta_n^{\mathbb K}\le_{\mathrm{Lt}}X_n^{\mathbb K}$.  Hence, for every
		$s\ge0$,
		\[
		\left(\frac{s}{\pi}\right)^{\beta/2}\E e^{-sX_n^{\mathbb K}}
		\le
		\left(\frac{s}{\pi}\right)^{\beta/2}\E e^{-s\Delta_n^{\mathbb K}}.
		\]
		Letting $s\to\infty$ and using \eqref{eq:kappa-definition} gives
		\begin{equation}\label{eq:gaussian-density-exact}
			p_n^{\mathbb K}(z)\le p_n^{\mathbb K}(0)
			\le\kappa_{n,\beta}.
		\end{equation}
		The ball of radius $\varepsilon$ in $\mathbb K\simeq\R^\beta$ has volume
		$\omega_\beta\varepsilon^\beta$, so
		\begin{equation}\label{eq:gaussian-small-ball-exact}
			\sup_{z\in\mathbb K}
			\Pp(|W_n^{\mathbb K}-z|\le\varepsilon)
			\le\min\{1,\omega_\beta\kappa_{n,\beta}\varepsilon^\beta\}.
		\end{equation}
		Together with \eqref{eq:kappa-asymptotic}, these estimates prove
		\eqref{eq:gaussian-density-headline} and
		\eqref{eq:gaussian-small-ball-headline}.
	\end{proof}
	
	For $\mathbb K=\C$, $\omega_2=\pi$ and $\kappa_{n,2}=n/\pi$, so the exact
	bound \eqref{eq:gaussian-small-ball-exact} becomes
	\begin{equation}\label{eq:pacc-quadratic}
		\sup_{z\in\C}
		\Pp\bigl(|\per G_n-z|\le\varepsilon\sqrt{n!}\bigr)
		\le\min\{1,n\varepsilon^2\}.
	\end{equation}
	
	\section{Gaussian perturbations}
	\label{sec:gaussian-perturbations}
	
	For additive Gaussian noise, the Laplace comparison holds at every center
	even in the presence of an arbitrary independent perturbation.  It therefore
	gives a uniform density bound.
	
	\begin{theorem}[Gaussian perturbation comparison]
		\label{thm:gaussian-perturbation}
		Let $Z\in\mathbb K^{n\times n}$ be an arbitrary random matrix independent
		of a standard $\mathbb K$-Ginibre matrix $G_n^{\mathbb K}$.  Then, for
		every $w\in\mathbb K$ and $s\ge0$,
		\begin{equation}\label{eq:gaussian-perturbation-laplace}
			\E e^{-s|\per_{\mathbb K}(Z+G_n^{\mathbb K})-w|^2}
			\le
			\E e^{-s|\per_{\mathbb K}G_n^{\mathbb K}|^2}.
		\end{equation}
		Moreover,
		$\per_{\mathbb K}(Z+G_n^{\mathbb K})/\sqrt{n!}$
		has a density $q$ satisfying
		\begin{equation}\label{eq:gaussian-perturbation-density}
			\|q\|_\infty\le p_n^{\mathbb K}(0)\le\kappa_{n,\beta}.
		\end{equation}
		Consequently,
		\begin{equation}\label{eq:gaussian-perturbation-small-ball}
			\sup_{w\in\mathbb K}
			\Pp\left(
			|\per_{\mathbb K}(Z+G_n^{\mathbb K})-w|
			\le\varepsilon\sqrt{n!}
			\right)
			\le
			\min\{1,\omega_\beta\kappa_{n,\beta}\varepsilon^\beta\}.
		\end{equation}
	\end{theorem}
	
	The same statement with positive row-dependent Gaussian scales follows by
	rescaling the rows.  For $\mathbb K=\R$ and deterministic $Z=tW$,
	Theorem~\ref{thm:gaussian-perturbation} proves the conjecture on gently
	perturbed Gaussian permanents, Conjecture~6 of Bouland, Datta, Fefferman, and
	Hern\'andez \cite{BoulandDattaFeffermanHernandez2025}.  The theorem does not
	require their restrictions that $t=O(n^{-1/2})$ and that $W$ have bounded
	entries.
	
	Equality holds in \eqref{eq:gaussian-perturbation-laplace} when $Z=0$ and
	$w=0$, and the comparison $\|q\|_\infty\le p_n^{\mathbb K}(0)$ is attained
	when $Z=0$.  This does not assert optimality of the explicit bound
	$p_n^{\mathbb K}(0)\le\kappa_{n,\beta}$.
	
	\begin{proof}
		If $g$ is a standard $\mathbb K$-Gaussian row, then for every
		$c\in\mathbb K^n$, $b\in\mathbb K$, and $s\ge0$,
		\begin{equation}\label{eq:noncentral-gaussian-kernel}
			\E e^{-s|gc+b|^2}
			=
			\left(1+\frac{2s}{\beta}\|c\|_2^2\right)^{-\beta/2}
			\exp\left(
			-\frac{s|b|^2}{1+2s\|c\|_2^2/\beta}
			\right)
			\le \E e^{-s|gc|^2}.
		\end{equation}
		
		First suppose $\mathbb K\in\{\R,\C\}$.  Conditioning on $Z$ reduces to
		deterministic $Z$.
		Condition on all but the first row and let $C^{(1)}$ be its permanental
		cofactor column.  Formula \eqref{eq:noncentral-gaussian-kernel}, with
		$b=Z_1C^{(1)}-w$, replaces the first row by a centered Gaussian row and removes
		the center $w$.  At each subsequent replacement, let $C^{(i)}$ denote the
		cofactor column determined by all other current rows.  The same formula with
		$b=Z_iC^{(i)}$ replaces row $i$ by a centered Gaussian row.  Iterating proves
		\eqref{eq:gaussian-perturbation-laplace}.  The quaternionic argument is
		given in Subsection~\ref{subsec:quaternion-row-replacement}.
		
		Put
		\[
		V=\frac{\per_{\mathbb K}(Z+G_n^{\mathbb K})}{\sqrt{n!}}.
		\]
		Setting $w=\sqrt{n!}\,u$ and replacing $s$ by $t/n!$ in
		\eqref{eq:gaussian-perturbation-laplace} gives, for every
		$u\in\mathbb K$ and $t\ge0$,
		\[
		\E e^{-t|V-u|^2}\le \E e^{-tX_n^{\mathbb K}}.
		\]
		Let $\mu$ be the law of $V$ and
		$K_t(x)=(t/\pi)^{\beta/2}e^{-t|x|^2}$.  Multiplying the preceding inequality
		by $(t/\pi)^{\beta/2}$ shows that
		\[
		(K_t*\mu)(u)\le p_n^{\mathbb K}(0)
		\qquad(u\in\mathbb K,t>0),
		\]
		where we used \eqref{eq:ginibre-density-laplace-limit} on the right.  For every
		nonnegative $f\in C_c(\mathbb K)$, Fubini's theorem and the approximate-identity
		property give
		\[
		\int f\,\dd\mu
		=\lim_{t\to\infty}\int f(u)(K_t*\mu)(u)\,\dd u
		\le p_n^{\mathbb K}(0)\int f(u)\,\dd u.
		\]
		Thus $\mu$ has a density $q$ with
		$\|q\|_\infty\le p_n^{\mathbb K}(0)$.  Equation
		\eqref{eq:gaussian-density-exact} proves
		\eqref{eq:gaussian-perturbation-density}, and integration over a ball of
		radius $\varepsilon$ proves
		\eqref{eq:gaussian-perturbation-small-ball}.
	\end{proof}
	
	\appendix
	
	\section{Failure of stronger comparisons}
	\label{app:stronger-orders}
	
	Following the standard terminology
	\cite[Sections~1.A and~3.A]{ShakedShanthikumar2007},
	$U\le_{\mathrm{st}}V$, the usual stochastic order, means
	$\Pp(U>t)\le\Pp(V>t)$ for every $t$.  For integrable random variables,
	$U\le_{\mathrm{cx}}V$, the convex order, means
	$\E f(U)\le\E f(V)$ for every convex function $f\colon\R\to\R$ for which
	the expectations exist.  The usual stochastic order compares magnitude,
	whereas convex order compares variability and in particular forces equal means.
	
	Since $x\mapsto e^{-sx}$ is decreasing and convex, either of the comparisons
	\[
	\Delta_n^{\mathbb K}\le_{\mathrm{st}}X_n^{\mathbb K}
	\qquad\text{or}\qquad
	X_n^{\mathbb K}\le_{\mathrm{cx}}\Delta_n^{\mathbb K}
	\]
	would imply $\Delta_n^{\mathbb K}\le_{\mathrm{Lt}}X_n^{\mathbb K}$.  We show
	that neither holds, already at $n=3$, for each
	$\mathbb K\in\{\R,\C,\Hh\}$.  The normalized variables from the introduction
	satisfy
	\[
	\E\Delta_n^{\mathbb K}=\E X_n^{\mathbb K}=1.
	\]
	If either usual stochastic comparison held, integration of the corresponding
	tail inequality would therefore force equality in distribution.  The laws do
	agree for $n\le2$.  Indeed, at $n=2$, write
	$G_2^{\mathbb K}=\left(\begin{smallmatrix}a&b\\c&d\end{smallmatrix}\right)$.
	Conditioning on the second row gives
	\[
	|\per_{\mathbb K}G_2^{\mathbb K}|^2
	\stackrel{\mathrm d}=\gamma_{1,\beta}(|c|^2+|d|^2)
	\stackrel{\mathrm d}=\gamma_{1,\beta}\gamma_{2,\beta},
	\]
	which is the determinant or Study-determinant benchmark.
	
	Over $\R$ and $\C$, their first two moments agree in every dimension.  A Wick expansion
	gives
	\[
	\E\Delta_n^2=\E X_n^2
	=
	\begin{cases}
		\binom{n+2}{2},&\F=\R,\\
		n+1,&\F=\C.
	\end{cases}
	\]
	Equivalently, the fourth-moment expansions of the determinant and permanent
	coincide: every surviving Wick term has positive total permutation sign.
	At $n=3$, the quaternionic second moments also agree, with
	$\E(\Delta_3^{\Hh})^2=\E(X_3^{\Hh})^2=35/16$.  This agreement fails at
	$n=4$:
	\[
	\E(\Delta_4^{\Hh})^2=\frac{315}{128},
	\qquad
	\E(X_4^{\Hh})^2=\frac{39}{16}=\frac{312}{128}.
	\]
	The third moments separate at $n=3$ in all three fields.  To make the
	calculation explicit, let $u,v\in\mathbb K^3$ be the last two rows and set
	\begin{align*}
		S_{\mathrm{per}}^{\mathbb K}
		&=\sum_{1\le i<j\le3}|u_iv_j+u_jv_i|^2,\\
		S_{\Delta}^{\mathbb K}
		&=
		\begin{cases}
			\displaystyle\sum_{1\le i<j\le3}|u_iv_j-u_jv_i|^2,
			&\mathbb K\in\{\R,\C\},\\[2mm]
			\displaystyle\|u\|_2^2\|v\|_2^2
			-\left|\sum_{j=1}^3u_j\overline{v_j}\right|^2,
			&\mathbb K=\Hh.
		\end{cases}
	\end{align*}
	Integrating out the first row gives, with the gamma variable independent of the
	two scales,
	\[
	X_3^{\mathbb K}\stackrel{\mathrm d}=
	\frac{\gamma_{1,\beta}S_{\mathrm{per}}^{\mathbb K}}{3!},
	\qquad
	\Delta_3^{\mathbb K}\stackrel{\mathrm d}=
	\frac{\gamma_{1,\beta}S_{\Delta}^{\mathbb K}}{3!}.
	\]
	For fixed $u$, write the $3\beta$ real coordinates of $v$ as
	$g/\sqrt\beta$, where $g$ is standard real Gaussian.  There are real symmetric
	positive semidefinite matrices $Q_{\mathrm{per}}(u)$ and $Q_\Delta(u)$ such that
	\[
	S_{\mathrm{per}}^{\mathbb K}
	=\frac1\beta g^TQ_{\mathrm{per}}(u)g,
	\qquad
	S_{\Delta}^{\mathbb K}
	=\frac1\beta g^TQ_\Delta(u)g.
	\]
	For any real symmetric $Q$,
	\begin{equation}\label{eq:gaussian-quadratic-third-moment}
		\E_g(g^TQg)^3
		=(\operatorname{tr}Q)^3
		+6\operatorname{tr}(Q)\operatorname{tr}(Q^2)
		+8\operatorname{tr}(Q^3).
	\end{equation}
	For either quadratic form, put
	\[
	A(Q)=\E_u(\operatorname{tr}Q)^3,
	\qquad
	B(Q)=\E_u[\operatorname{tr}(Q)\operatorname{tr}(Q^2)],
	\qquad
	C(Q)=\E_u\operatorname{tr}(Q^3).
	\]
	A direct Wick average in the real coordinates of $u$ gives
	\[
	\begin{array}{c|ccc|ccc}
		&\multicolumn{3}{c|}{Q_\Delta}
		&\multicolumn{3}{c}{Q_{\mathrm{per}}}\\
		\mathbb K&A&B&C&A&B&C\\ \hline
		\R  &840   &420  &210&840   &420  &222\\
		\C  &3840  &960  &240&3840  &960  &264\\
		\Hh &21504 &2688 &336&21504 &2688 &366.
	\end{array}
	\]
	Thus
	$\E(S^{\mathbb K})^3=\beta^{-3}(A+6B+8C)$ for the corresponding scale.
	Since
	$\E\gamma_{1,1}^3=15$, $\E\gamma_{1,2}^3=6$, and
	$\E\gamma_{1,4}^3=3$, division by $(3!)^3$ gives
	\[
	\begin{array}{c@{\qquad}c@{\qquad}c}
		\mathbb K
		& \E(\Delta_3^{\mathbb K})^3
		& \E(X_3^{\mathbb K})^3\\ \hline
		\R & 350 & 1070/3\\
		\C & 40  & 122/3\\
		\Hh & 35/4 & 845/96.
	\end{array}
	\]
	Together with the equal means, these discrepancies rule out usual stochastic
	domination in either direction for every $\mathbb K$.
	The third moments also rule out the natural convex-order strengthening of
	the Laplace-transform order result proved in this paper.  Namely,
	$X_3^{\mathbb K}\le_{\mathrm{cx}}\Delta_3^{\mathbb K}$, applied to the convex
	function $f(t)=|t|^3$, would imply
	$\E(X_3^{\mathbb K})^3\le\E(\Delta_3^{\mathbb K})^3$, contrary to the table.
	
	Figure~\ref{fig:order-separation} displays the same qualitative separation over
	$\R$, $\C$, and $\Hh$: the distribution functions cross, while the
	Laplace-transform differences retain the same sign.  Indeed, for nonnegative $U,V$,
	\[
	\E e^{-sU}-\E e^{-sV}
	=s\int_0^\infty e^{-st}\bigl(F_U(t)-F_V(t)\bigr)\,dt,
	\]
	so the exponential kernel averages over the crossing rather than comparing the
	CDFs pointwise.
	
	\begin{figure}[htbp]
		\centering
		\begin{subfigure}[t]{0.48\textwidth}
			\centering
			\includegraphics[width=\linewidth]{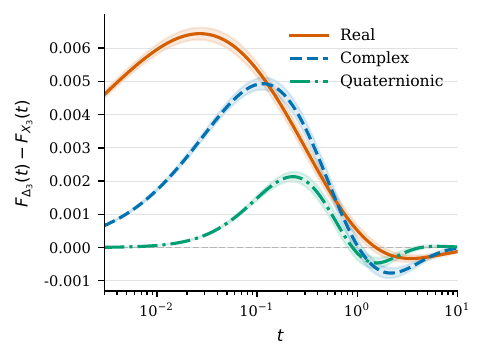}
			\caption{Difference of CDFs, determinantal benchmark minus squared permanent:
				the sign changes.}
		\end{subfigure}
		\hfill
		\begin{subfigure}[t]{0.48\textwidth}
			\centering
			\includegraphics[width=\linewidth]{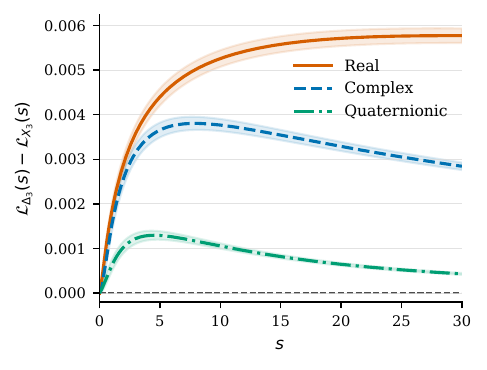}
			\caption{Difference of Laplace transforms, determinantal benchmark minus
				squared permanent: the sign is nonnegative.}
		\end{subfigure}
		\caption{Monte Carlo comparison of the normalized determinantal benchmark
			$\Delta_3^{\mathbb K}$ and normalized squared permanent $X_3^{\mathbb K}$
			for $3\times3$ Ginibre matrices over $\mathbb K\in\{\R,\C,\Hh\}$.  Over
			$\Hh$, the benchmark uses the Study determinant and the permanent is
			row-ordered.  The estimates use $3\mathord\times10^6$ samples in each
			field; shaded bands show two standard errors.  One Gaussian row is
			integrated out analytically.}
		\label{fig:order-separation}
	\end{figure}
	\FloatBarrier
	
	\section{The Cayley determinant}
	\label{app:cayley-determinant}
	
	For $A\in\Hh^{n\times n}$, the signed row-ordered expression
	\[
	\Cdet A
	:=\sum_{\sigma\in S_n}\operatorname{sgn}(\sigma)
	a_{1,\sigma(1)}a_{2,\sigma(2)}\cdots a_{n,\sigma(n)}
	\]
	is known as the \emph{Cayley determinant} or \emph{row-determinant}
	\cite{ChienRasmussenSinclair2003,ArvindSrinivasan2010}.  It is distinct from
	the Study determinant.  For example, with the usual quaternion units,
	\[
	U=\frac1{\sqrt2}
	\begin{pmatrix}\mathbf i&\mathbf j\\ \mathbf j&\mathbf i\end{pmatrix}
	\quad\text{is unitary, but}\quad
	\Cdet U=0,
	\]
	so $\Sdet U=1$; this example also appears in
	\cite{CohenDeLeo2000}.  Conversely,
	\[
	A=\begin{pmatrix}\mathbf i&\mathbf j\\ \mathbf i&\mathbf j\end{pmatrix}
	\quad\text{is singular, but}\quad
	\Cdet A
	=\mathbf i\mathbf j-\mathbf j\mathbf i=2\mathbf k.
	\]
	Thus the two determinants neither agree in magnitude nor control one another
	pointwise.  Instead, $\Sdet(A)^2$ is the Moore determinant of $A^*A$, as
	recalled in Section~\ref{sec:determinantal-benchmark}.
	\par\medskip
	
	\begin{theorem}[Study--Cayley determinant comparison]
		\label{thm:study-cayley-comparison}
		For every $n\ge1$,
		\[
		\Sdet(G_n^{\Hh})^2
		\le_{\mathrm{Lt}}
		\left|\Cdet G_n^{\Hh}\right|^2.
		\]
	\end{theorem}
	
	\begin{proof}
		The claim is immediate at $n=1$.  For $n\ge2$, let $B$ be an
		$(n-1)\times n$ standard quaternionic Gaussian matrix.  Write $D(B)$ for
		its signed Cayley cofactor column,
		\[
		D(B)_j=(-1)^{j-1}\Cdet(B_{-j}),
		\qquad 1\le j\le n,
		\]
		and set $D_n=D(B)$ and $D_1=(1)$.
		For every row $w\in\Hh^n$, expansion along the first row gives
		\[
		\Cdet\begin{pmatrix}w\\B\end{pmatrix}
		=\sum_{j=1}^nw_jD(B)_j.
		\]
		We briefly indicate why the proof of Fourier coordinate compression,
		Proposition~\ref{prop:compression}, applies to $D_n$.  For a permutation
		matrix $P$,
		\[
		D(BP)=\operatorname{sgn}(P)P^{-1}D(B).
		\]
		Negating one Gaussian row shows that $D_n\stackrel{\mathrm d}=-D_n$, so
		column-permutation invariance of $B$ makes the coordinates of $D_n$
		exchangeable.  Left multiplication of the first row of $B$ by a unit
		quaternion gives the same unit-scalar invariance used in the main proof.
		
		For the two-coordinate step, write $B=[X,Y,R]$ and $w=\overline z$.  Expansion
		along the fixed first row gives
		\begin{align*}
			\sum_{j=1}^nw_j(D_n)_j
			={}&w_1\Cdet[Y,R]
			-w_2\Cdet[X,R]\\
			&+\sum_{j=3}^n(-1)^{j-1}w_j
			\Cdet[X,Y,R_{-j}].
		\end{align*}
		After identifying quaternionic vectors with their underlying real vectors,
		this has exactly the $F_{T,L}(a,b)$ form from the main proof with
		$(a,b)=(w_1,-w_2)$, to which Gaussian two-coordinate compression applies
		unchanged.  Conditioning on the last Gaussian row also gives the same
		nonnegative conditional characteristic function.  Hence
		\[
		0\le \E e^{\ii(z,D_n)_{\R}}
		\le \E e^{\ii(\|z\|_2e_1,D_n)_{\R}}.
		\]
		
		The proof of Theorem~\ref{thm:gaussian-cofactor-gamma} now applies without
		further changes.  The first coordinate of $D_{n+1}$ has the law of
		$\Cdet G_n^{\Hh}$, and its expansion along an independent first
		row has the same one-coordinate Gaussian recursion.  The resulting scalar
		comparison is
		\[
		\Pi_{n,4}\le_{\mathrm{Lt}}\|D_n\|_2^2.
		\]
		A final first-row expansion gives, with an independent factor on the right,
		\[
		\left|\Cdet G_n^{\Hh}\right|^2
		\stackrel{\mathrm d}=\gamma_{1,4}\|D_n\|_2^2.
		\]
		Multiplying the preceding comparison by $\gamma_{1,4}$ and using
		Proposition~\ref{prop:study-determinant-benchmark} proves the theorem.
	\end{proof}
	
	At $n=2$, the comparison is equality in distribution.  Indeed, if
	$G_2^{\Hh}=\left(\begin{smallmatrix}a&b\\c&d\end{smallmatrix}\right)$, then
	conditioning on $(c,d)$ and using quaternionic Gaussian rotational invariance
	gives
	\[
	|ad-bc|^2
	\stackrel{\mathrm d}=\gamma_{1,4}(|c|^2+|d|^2)
	\stackrel{\mathrm d}=\gamma_{1,4}\gamma_{2,4}
	\stackrel{\mathrm d}=\Sdet(G_2^{\Hh})^2.
	\]
	Equality in distribution is special to low dimension.  One moment identity
	does persist: the Cayley determinant and row-ordered permanent have the same
	fourth moment in every dimension,
	\begin{equation}
		\E\left|\Cdet G_n^{\Hh}\right|^4
		=\E\left|\per_{\Hh}G_n^{\Hh}\right|^4.
		\label{eq:cayley-permanent-fourth-moment}
	\end{equation}
	To see this, expand either side using four permutations.  A term can have
	nonzero expectation only if every independent matrix entry occurs an even
	number of times.  For such a term, the symmetric differences of the first two
	and last two permutation matchings coincide.  The relative sign of two
	permutations is determined by the cycles in their symmetric difference, so the
	product of the four permutation signs is $1$.  The surviving terms in the two
	expansions therefore agree term by term.
	
	Their magnitude distributions nevertheless differ already at $n=3$.  Let
	$u,v\in\Hh^3$ be the last two rows of $G_3^{\Hh}$ and set
	\[
	S_{\pm}:=\sum_{1\le i<j\le3}|u_i v_j\pm u_jv_i|^2.
	\]
	Conditioning on $u,v$ and integrating out the first row gives, with the gamma
	variable independent of $S_{\pm}$,
	\[
	|\per_{\Hh}G_3^{\Hh}|^2\stackrel{\mathrm d}=\gamma_{1,4}S_+,
	\qquad
	|\Cdet G_3^{\Hh}|^2\stackrel{\mathrm d}=\gamma_{1,4}S_-.
	\]
	Here $S_+=S_{\mathrm{per}}^{\Hh}$ from Appendix~\ref{app:stronger-orders}.
	For fixed $u$, write the twelve real coordinates of $v$ as $g/2$ and define
	$Q_-(u)$ by $S_-=\frac14g^TQ_-(u)g$.  In the notation used with
	\eqref{eq:gaussian-quadratic-third-moment}, Wick averaging gives
	\[
	A(Q_-)=21504,
	\qquad B(Q_-)=2688,
	\qquad C(Q_-)=354.
	\]
	Thus
	\[
	\E S_-^3
	=\frac{21504+6\cdot2688+8\cdot354}{4^3}
	=\frac{2529}{4},
	\qquad
	\E S_+^3=\frac{2535}{4},
	\]
	where the last value is recorded in Appendix~\ref{app:stronger-orders}.
	Since $\E\gamma_{1,4}^3=3$, these values and the Study-determinant product
	formula yield the strict three-way discrepancy
	\[
	\E\left(\frac{\Sdet(G_3^{\Hh})^2}{3!}\right)^3
	=\frac{840}{96}
	<\E\left(\frac{|\Cdet G_3^{\Hh}|^2}{3!}\right)^3
	=\frac{843}{96}
	<\E\left(\frac{|\per_{\Hh}G_3^{\Hh}|^2}{3!}\right)^3
	=\frac{845}{96}.
	\]
	In particular, the magnitudes of the Cayley determinant and permanent do not
	have the same distribution.
	
	At $n=4$, the Study--Cayley separation is visible already in the second moment.
	Indeed, equation~\eqref{eq:cayley-permanent-fourth-moment} and the moments
	recorded in Appendix~\ref{app:stronger-orders} give
	\[
	\E\left(\frac{|\Cdet G_4^{\Hh}|^2}{4!}\right)^2
	=\frac{39}{16}=\frac{312}{128}
	<\frac{315}{128}
	=\E\left(\frac{\Sdet(G_4^{\Hh})^2}{4!}\right)^2.
	\]
	Thus the equality at $n=2$ and the fourth-moment identity above are both
	special coincidences rather than equality of the three laws.
	
	\section{Gaussian comparison for independent rows}
	\label{app:independent-rows}
	
	This appendix gives a sufficient projection condition under which comparison
	with the Gaussian model extends to matrices with independent, not necessarily
	Gaussian, rows.  The resulting normalization is the exact second-moment scale.
	We do not pursue comparisons at an arbitrary smaller scale: a fixed loss in
	each row becomes an exponential loss for the permanent.  We also record the
	quaternionic row-replacement details used here and in
	Section~\ref{sec:gaussian-perturbations}.
	
	Fix $\mathbb K\in\{\R,\C,\Hh\}$ and put
	$\beta=\dim_{\R}\mathbb K$.
	
	\subsection{Projection criterion}
	
	\begin{theorem}[Independent-row comparison]
		\label{thm:independent-row-comparison}
		Let $A\in\mathbb K^{n\times n}$ have independent rows
		$R_1,\ldots,R_n$.  Suppose that, for some $\sigma_1,\ldots,\sigma_n>0$,
		\begin{equation}\label{eq:row-projection-criterion}
			\E R_i^*R_i=\sigma_i^2I_n,
			\qquad
			\gamma_{1,\beta}
			\le_{\mathrm{Lt}}
			\frac{|R_ic|^2}{\sigma_i^2\|c\|_2^2}
			\quad(c\in\mathbb K^n\setminus\{0\}).
		\end{equation}
		Then
		\begin{equation}\label{eq:independent-row-laplace-comparison}
			\Delta_n^{\mathbb K}
			\le_{\mathrm{Lt}}X_n^{\mathbb K}
			\le_{\mathrm{Lt}}
			\frac{|\per_{\mathbb K}A|^2}
			{n!\prod_{i=1}^n\sigma_i^2}.
		\end{equation}
		Consequently, for every $\varepsilon>0$,
		\begin{equation}\label{eq:independent-row-small-ball}
			\Pp\left(
			|\per_{\mathbb K}A|
			\le\varepsilon\sqrt{n!}\prod_{i=1}^n\sigma_i
			\right)
			\le
			\min\left\{1,
			\left(\frac{2\pi e}{\beta}\right)^{\beta/2}
			\kappa_{n,\beta}\varepsilon^\beta
			\right\}.
		\end{equation}
	\end{theorem}
	
	Independence of the rows and the isotropy condition give
	\[
	\E|\per_{\mathbb K}A|^2
	=n!\prod_{i=1}^n\sigma_i^2,
	\]
	so the normalization in \eqref{eq:independent-row-laplace-comparison} is the exact
	second-moment scale.
	
	\begin{proof}
		First suppose $\mathbb K\in\{\R,\C\}$.  The second condition in
		\eqref{eq:row-projection-criterion} is equivalent to
		\[
		\E e^{-s|R_ic|^2}
		\le
		\E e^{-s\sigma_i^2|g_ic|^2}
		\qquad(c\in\mathbb K^n,\ s\ge0),
		\]
		where $g_i$ is a standard $\mathbb K$-Gaussian row.  For $0\le k\le n$,
		let $A^{(k)}$ have row $\sigma_i g_i$ for $i\le k$ and row $R_i$ for
		$i>k$.  For $k\ge 1$, condition on every row except row $k$.
		If $C^{(k)}$ is the resulting
		permanental cofactor column, then
		\[
		\per A^{(k-1)}=R_kC^{(k)},
		\qquad
		\per A^{(k)}=\sigma_k g_kC^{(k)}.
		\]
		Applying the projection comparison conditionally and iterating over the
		rows gives
		\[
		\E e^{-s|\per A|^2}
		\le
		\E e^{-s(\prod_{i=1}^n\sigma_i^2)|\per G_n|^2}.
		\]
		The same row-replacement comparison over $\Hh$ is proved in
		Subsection~\ref{subsec:quaternion-row-replacement}.  Thus, in every field,
		normalization gives
		$X_n^{\mathbb K}\le_{\mathrm{Lt}}
		|\per_{\mathbb K}A|^2/(n!\prod_{i=1}^n\sigma_i^2)$.
		The first comparison in \eqref{eq:independent-row-laplace-comparison} is
		\eqref{eq:determinant-permanent-comparison} over $\R$ and $\C$, and
		Theorem~\ref{thm:quaternion-laplace} over $\Hh$.
		
		The small-ball estimate is
		\eqref{eq:laplace-chernoff-small-ball} applied with
		$Y=|\per_{\mathbb K}A|^2/(n!\prod_{i=1}^n\sigma_i^2)$.
	\end{proof}
	
	If the projection comparison held with a scale $\tau_i^2$ in place of
	$\sigma_i^2$, comparison of right derivatives at zero would give
	$\tau_i^2\le \E|R_ic|^2/\|c\|_2^2=\sigma_i^2$.  Thus $\sigma_i^2$ is the largest
	admissible Gaussian scale.
	
	There is a simple radial source of rows satisfying
	\eqref{eq:row-projection-criterion}.
	
	\begin{proposition}[Radial rows]
		\label{prop:radial-rows}
		Let $U$ be uniform on the unit sphere of $\mathbb K^n$, let $Q\ge0$ be
		independent of $U$, and suppose that
		\[
		\E Q=n,
		\qquad
		\gamma_{n,\beta}\le_{\mathrm{Lt}}Q.
		\]
		Then $R=\sigma\sqrt Q\,U$ satisfies
		\eqref{eq:row-projection-criterion} at scale $\sigma^2$.
	\end{proposition}
	
	\begin{proof}
		Rotational invariance gives $\E U^*U=I_n/n$, so
		$\E R^*R=\sigma^2I_n$.  A standard $\mathbb K$-Gaussian row has the
		radial decomposition
		\[
		g_n\stackrel{\mathrm d}=\sqrt{\gamma_{n,\beta}}\,U,
		\]
		with independent radius and direction.  Therefore, for every
		$c\in\mathbb K^n$ and $s\ge0$,
		\[
		\E e^{-s|Rc|^2}
		=\E_U\E_Qe^{-s\sigma^2Q|Uc|^2}
		\le
		\E_U\E_{\gamma_{n,\beta}}
		e^{-s\sigma^2\gamma_{n,\beta}|Uc|^2}
		=\E e^{-s\sigma^2|g_nc|^2}.
		\]
	\end{proof}
	
	\begin{remark}[Examples]
		For a mean-one nonnegative random variable $V$, use the
		reliability-theoretic notation $V\in\mathcal L_\alpha$ when
		$\operatorname{Gamma}(\alpha,1/\alpha)\le_{\mathrm{Lt}}V$
		\cite{Klefsjo1983,Lin1998,Klar2002}.  Thus the second condition in
		\eqref{eq:row-projection-criterion} says that every normalized projected power
		belongs to $\mathcal L_{\beta/2}$, while the hypothesis of
		Proposition~\ref{prop:radial-rows} says equivalently
		that $Q/n$ belongs to $\mathcal L_{\beta n/2}$.
		
		The proposition includes the following natural examples.
		
		If $Q=n$, then $\sqrt n\,U$ is the normalized uniform sphere row; the
		Laplace comparison follows immediately from Jensen's inequality.  If $B$
		is uniform on the unit ball, then
		$\sqrt{n+2/\beta}\,B$ also satisfies the proposition.  Indeed, write
		$B=\sqrt T\,U$ with
		$T\sim\operatorname{Beta}(\beta n/2,1)$.  For an independent
		$S\sim\operatorname{Gamma}(\beta n/2+1,2/\beta)$,
		the beta--gamma identity gives
		$TS\stackrel{\mathrm d}=\gamma_{n,\beta}$ and
		$\E S=n+2/\beta$; conditional Jensen proves the required order.
		
		More generally, if
		\[
		Q\sim\operatorname{Gamma}(a,n/a),
		\qquad a\ge\frac{\beta n}{2},
		\]
		then direct comparison of Laplace transforms proves the hypothesis.  This
		family interpolates between the Gaussian radius and the sphere.  Another
		large family is obtained from any mean-preserving contraction of the
		Gaussian squared radius: if
		$Q=\E[\gamma_{n,\beta}\mid\mathcal G]$, then conditional Jensen gives
		$\gamma_{n,\beta}\le_{\mathrm{Lt}}Q$.
	\end{remark}
	
	\subsection{Quaternionic row replacement}
	\label{subsec:quaternion-row-replacement}
	
	We prove the quaternionic row-replacement statements invoked in the proofs of
	Theorems~\ref{thm:independent-row-comparison}
	and~\ref{thm:gaussian-perturbation}.  The only additional issue is
	that replacing an interior row places quaternionic coefficients on both sides
	of its entries.
	
	\begin{lemma}[Quaternionic Gaussian left-right symmetry]
		\label{lem:quaternion-gaussian-left-right}
		For fixed $d_1,\ldots,d_p\in\Hh$ and independent standard quaternionic
		Gaussians $g_1,\ldots,g_p,h_1,\ldots,h_p$,
		\[
		\sum_{j=1}^pg_jd_j
		\stackrel{\mathrm d}=
		\sum_{j=1}^pd_jh_j.
		\]
	\end{lemma}
	
	\begin{proof}
		Both sides are centered real Gaussian vectors with covariance
		$\frac14\sum_j|d_j|^2I_4$.
	\end{proof}
	
	For a quaternionic matrix with rows $R_1,\ldots,R_n$, let
	\[
	\mathcal S_k
	=\{S\subseteq\{1,\ldots,n\}:|S|=n-k+1\}
	\qquad(1\le k\le n+1)
	\]
	and, for $S\in\mathcal S_k$, define
	\[
	P_k(S)=
	\sum_{\substack{\pi:\{k,\ldots,n\}\to S\\
			\pi\ {\rm bijective}}}
	R_{k,\pi(k)}\cdots R_{n,\pi(n)},
	\qquad
	P_{n+1}(\varnothing)=1.
	\]
	Thus $P_1(\{1,\ldots,n\})=\per_{\Hh}A$ and
	\[
	P_k(S)=\sum_{j\in S}R_{k,j}P_{k+1}(S\setminus\{j\}).
	\]
	For coefficients $a=(a_S)_{S\in\mathcal S_k}$, write
	$P_k[a]=\sum_{S\in\mathcal S_k}P_k(S)a_S$.
	
	\begin{lemma}[One-step suffix contraction]
		\label{lem:quaternion-suffix-contraction}
		Suppose row $R_k$ satisfies
		\[
		\E e^{-s|R_kc|^2}
		\le \E e^{-s\sigma_k^2|g_kc|^2}
		\qquad(c\in\Hh^n,\ s\ge0),
		\]
		where $g_k$ is a standard quaternionic Gaussian row.  Let
		$a=(a_S)_{S\in\mathcal S_k}$ be random coefficients independent of
		$R_k,\ldots,R_n$.  Conditional on $R_{k+1},\ldots,R_n$ and $a$, fresh
		standard quaternionic Gaussians define coefficients
		$a'=(a'_T)_{T\in\mathcal S_{k+1}}$, independent of the remaining rows, such
		that
		\[
		\E_{R_k}e^{-s|P_k[a]|^2}
		\le
		\E e^{-s|P_{k+1}[a']|^2}
		\qquad(s\ge0).
		\]
		If $R_k=\sigma_kg_k$ is Gaussian, equality holds.
	\end{lemma}
	
	\begin{proof}
		Conditional on the remaining rows and $a$, put
		\[
		D_j=
		\sum_{\substack{S\in\mathcal S_k\\j\in S}}
		P_{k+1}(S\setminus\{j\})a_S.
		\]
		Then $P_k[a]=\sum_jR_{k,j}D_j$.  The projection comparison and
		Lemma~\ref{lem:quaternion-gaussian-left-right} give
		\[
		\E_{R_k}e^{-s|P_k[a]|^2}
		\le
		\E_h e^{-s\sigma_k^2|\sum_jD_jh_j|^2}.
		\]
		For $T\in\mathcal S_{k+1}$, define
		\[
		a'_T=\sigma_k\sum_{j\notin T}a_{T\cup\{j\}}h_j.
		\]
		Regrouping shows that
		$\sigma_k\sum_jD_jh_j=P_{k+1}[a']$, proving the claim.
	\end{proof}
	
	\begin{proof}[Quaternionic row-replacement proofs]
		For Theorem~\ref{thm:independent-row-comparison}, start with the single
		coefficient $a^{(1)}_{\{1,\ldots,n\}}=1$ and apply
		Lemma~\ref{lem:quaternion-suffix-contraction} successively.  The fresh
		Gaussians ensure at each step that the new coefficients are independent of
		the remaining rows.  Hence
		\[
		\E e^{-s|\per_{\Hh}A|^2}
		\le
		\E e^{-s|P_{n+1}[a^{(n+1)}]|^2}.
		\]
		Applying the same construction to rows $\sigma_i g_i$ makes every step an
		equality and gives the same terminal coefficients.  Since the $\sigma_i$
		are positive real scalars,
		\[
		\E e^{-s|\per_{\Hh}A|^2}
		\le
		\E e^{-s(\prod_{i=1}^n\sigma_i^2)|\per_{\Hh}G_n^{\Hh}|^2},
		\]
		as required.
		
		For Theorem~\ref{thm:gaussian-perturbation}, condition on $Z$.  At the
		first suffix step, write
		$\per_{\Hh}(Z+G_n^{\Hh})-w=\sum_jg_{1,j}D_j+b$, where
		$b=\sum_jZ_{1,j}D_j-w$.  The noncentral Gaussian formula
		\eqref{eq:noncentral-gaussian-kernel} removes $b$, and
		Lemma~\ref{lem:quaternion-gaussian-left-right} performs the first suffix
		contraction.  Each later row $Z_i+g_i$ satisfies the same centered
		projection comparison by \eqref{eq:noncentral-gaussian-kernel}, so the
		remaining suffix contractions end at the standard quaternionic Ginibre
		permanent.  This proves
		\eqref{eq:gaussian-perturbation-laplace}.
	\end{proof}
	
	\section*{Acknowledgments}
	The authors discussed ideas with ChatGPT and used Codex to assist with writing.
	We thank Claude for introducing us to \cite{GenestOuimetRichards2024}
	during a discussion of Laplace-transform order.
	The authors are responsible for all errors.

	\bibliographystyle{alpha}
	\bibliography{references}
	
\end{document}